\documentclass[10pt,a4paper,twoside]{article}

\usepackage{amsmath,amsfonts,amssymb,latexsym}
\usepackage{amsthm}
\usepackage{verbatim,soul}
\usepackage[english]{babel}
\usepackage{a4wide}
\usepackage{pstricks,graphicx}
\usepackage[small,euler-digits,icomma,OT1,T1]{eulervm}
\usepackage{hyperref}

\theoremstyle{plain}
\newtheorem{theorem}{Theorem}[section]
\newtheorem{proposition}{Proposition}[section]
\newtheorem{lemma}{Lemma}[section]
\newtheorem{corollary}{Corollary}[section]
\theoremstyle{definition}
\newtheorem{definition}{Definition}[section]
\theoremstyle{remark}
\newtheorem{remark}{Remark}[section]

\newcommand{\verteq}{\rotatebox{90}{$\,=$}}

\title{Strong existence and uniqueness of the stationary distribution for a stochastic inviscid dyadic model}

\author{Luisa Andreis\thanks{Dipartimento di Matematica, Universit\`a degli Studi di Padova, via Trieste 63, 35121 Padova (Italy); e-mail addresses: $\{$andreis, barbato, proz$\}$@math.unipd.it} \and David Barbato\footnotemark[1] \and Francesca Collet\thanks{Dipartimento di Matematica, Alma Mater Studiorum Universit\`a di Bologna, piazza di Porta San Donato 5, 40126 Bologna (Italy); e-mail address: francesca.collet@unibo.it} \and Marco Formentin\thanks{Dipartimento di Fisica e Astronomia ``Galileo Galilei'', Universit\`a degli Studi di Padova, via Marzolo 8, 35131 Padova (Italy); e-mail address: marco.formentin@rub.de} \and Luigi Provenzano\footnotemark[1]}

\begin{document}

\maketitle

\begin{abstract}
\noindent 
We consider an inviscid stochastically forced dyadic model, where the additive noise acts only on the first component. We prove that a strong solution for this problem exists and is unique by means of uniform energy estimates. Moreover, we exploit these results to establish strong existence and uniqueness of the stationary distribution.

\vspace{0.3cm}

\end{abstract}

\section{Introduction}

From a physical point of view, turbulence is characterized by the presence of an energy flow which, through local nonlinear interactions, transfers all the energy of the velocity field from larger to smaller scales. The energy is then dissipated by a viscous term. The energy-cascade mechanism is a phenomenon widely observed in fluid dynamics experiments and at present lacking rigorous mathematical understanding. 

Dyadic (or shell) models have been studied as toy models for the Euler and Navier-Stokes equations (inviscid and viscous case respectively). Even though much simpler, they display peculiar features of the nonlinear structure of fluid dynamic equations. In particular, they mimic the scale by scale local transfer of energy occurring in three-dimensional turbulent flows. Among the most remarkable applications it is worth to mention \cite{Tao14}, in which a dyadic model type approximation is used to show the emergence of blow-up for an averaged Navier-Stokes system.\\
The dyadic model was early introduced by Desnianskii and Novikov in 1974 \cite{DeNo74} and, in recent times, independently reintroduced by Katz and Pavlovi\'c in 2005 \cite{KaPa05}. Since then, it has met the interest of a wide scientific community and several versions have been extensively investigated: viscous \cite{BaMoRo11, BeFe12}; inviscid \cite{BaMo13, BaMoRo11, BeFe12, ChCoFrSh08, ChFrPa10, KiZl05}; stochastically forced, either with additive \cite{FrG-HVi, Rom14} or multiplicative noise \cite{BaFlMo10, BaFlMo11, BaMo13stochastic, Bia13}.

The present paper is concerned with the study of the following stochastically driven shell model
\begin{equation}\label{model_stochastic}
\left\{\begin{array}{l l }
du_0=-u_0u_1dt+\sigma dW(t) &\\
du_j= \left( -2^{cj}u_ju_{j+1}+2^{c(j-1)}u_{j-1}^2 \right) dt & \mbox{ for } j\geq1 \\
u(0) = \underline{u}, &   
\end{array}\right.
\end{equation}
with $t \in [0,T]$, $\sigma \in \mathbb{R}_+$, $c \, \in \, [1,3]$, $\underline{u} \in \ell^2$ , $\underline{u}_j\geq 0$ for every $j\geq 1$, and where $\{W(t): t \geq 0\}$ is a one dimensional Brownian motion.  \\
The equations describe an inviscid dyadic system, whose peculiarity is that the additive noise affects only the first component. The random perturbation is therefore one dimensional unlike what is treated in \cite{BaFlMo10, BaFlMo11, BaMo13stochastic, Bia13, Rom14} where the noise acts on all the components and is indeed infinite dimensional. \\
The constant $c$ is an intermittency parameter and, roughly speaking, represents the velocity of the energy transfer from shell to shell. The range of values $c \in \left[ 1, \frac{5}{2} \right]$ is essentially the one corresponding, within the simplification of the model, to the 3D Euler equations. The range arises from scaling arguments applied to the nonlinear term \cite{ChFr09}.\\
Observe that the peculiar expression of the quadratic terms in \eqref{model_stochastic} provides the \emph{formal} conservation property
\[
\sum_{j=0}^{+\infty} u_j \left( 2^{c(j-1)}u_{j-1}^2 -2^{cj}u_ju_{j+1} \right) = 0 
\] 
that gives an a priori bound on the solution $u$ in $ \ell^2$ uniformly in $c$. Additionally, the fact that the state variables lie on the positive real half-line forces the energy to move from lower to higher wave numbers only. 

Model \eqref{model_stochastic} has been introduced in \cite{FrG-HVi}, where the authors establish that it admits invariant martingale solutions. In our paper we focus on strong solutions and prove that a strong stationary distribution exists and is unique in the class of positive solutions. We would like to stress here that the monotonicity of the energy flow in dyadic models allows to derive our statements by using arguments and techniques that are more immediate than those used for the stochastic 2D Navier-Stokes problem \cite{HaMa06,HaMa08}. In particular, the key point is having at hand a sort of pathwise contraction property \eqref{contrazione} of the dynamics  that, in turn, enables to get uniqueness of the stationary distribution almost straightforwardly. An analogue of \eqref{contrazione} is only known to be valid for Burgers equations \cite{Bor13}; whereas, it is believed to be false for Euler equations.

More in detail, our manuscript is organized as follows. \\
\emph{Section~\ref{pathwise}.} We deal with existence and uniqueness of pathwise solutions for problem~\eqref{model_stochastic}. We start by showing that there exists a uniform energy bound, which ensures the existence of a global pathwise solution (see Subsection~\ref{subsect:existence}). In Subsection~\ref{subsect:regularity} we provide a regularity condition satisfied by the trajectories. The proof follows from a modification of an approach used in \cite{FrG-HVi}. Our idea consists in replacing expectations by integrals over a time interval so to get a pathwise, rather than on average, property. The statements in Subsections~\ref{subsect:existence} and \ref{subsect:regularity} are valid over the entire range of $c \in [1,3]$. At the end of the section we show continuity with respect to positive initial conditions and noise in the range of $c \in [1,3)$ (see Subsection~\ref{subsect:continuity}). The derivation of this result makes crucial use of the solution regularity obtained in the previous subsection. With this in hand, we straightforwardly get pathwise uniqueness of the solution. \\
\emph{Section~\ref{sect:strong:sol}.} This section is devoted to demonstrating existence and uniqueness in a strong sense both of the solution and of the invariant distribution. We begin by proving that the continuity of the unique pathwise solution of \eqref{model_stochastic} guarantees the adaptability of the trajectory with respect to the filtration generated by the initial datum and the Brownian motion, providing strong existence and uniqueness. At this stage, having uniqueness in law (due to the classical Yamada and Watanabe theorem) and the existence of a weak stationary solution (provided by \cite{FrG-HVi}), we are also able to ensure that a positive strong statistically stationary solution of \eqref{model_stochastic} exists. As for the uniqueness of this statistically invariant state, the proof is based on an optimal transport argument. Roughly speaking, we introduce a distance between probability measures as a cost function to be minimized. Then, we show that if two different stationary distributions existed,  the minimality prescribed by Kantorovich's formulation of the problem would be violated. The statements in Section~\ref{sect:strong:sol} are valid in the restricted range of $c \in [1,3)$.\\[-.4cm]

To our knowledge we provide the first result of strong existence and pathwise uniqueness for a stochastic shell model in the inviscid  case.

\section{Pathwise solution}\label{pathwise}

In this section we consider the deterministic system obtained from \eqref{model_stochastic} by fixing a realization of the Brownian motion. Observe that in system \eqref{model_stochastic} the stochastic integral appears only in the equation 
$du_0(t)=-u_0(t)u_1(t)dt+\sigma dW(t)$, which is equivalent to 
$$u_0(t)=u_0(0)-\int_0^tu_1(s)u_0(s)ds+\sigma W(t) \ \ a.s..$$
Therefore, for every $\omega \in \Omega$ such that $w(\cdot):=W(\cdot,\omega) \in C([0,T],\mathbb{R})$ with $w(0)=0$, it is natural to define the following infinite dimensional deterministic system:
\begin{equation}\label{model}
\left\{\begin{array}{l l }
\displaystyle{u_0(t)=u_0(0)-\int_0^tu_1(s)u_0(s)ds+\sigma w(t)} &\\
du_j= \left( -2^{cj}u_ju_{j+1}+2^{c(j-1)}u_{j-1}^2 \right) dt & \mbox{ for } j\geq1 \\
u(0) = \underline{u}\, , &   
\end{array}\right.
\end{equation}
where $t \in [0,T]$, $\sigma \in \mathbb{R}_+$, $c \, \in \, [1,3]$, $\underline{u} \in H_+$ and $w \in C([0,T],\mathbb{R})$ with $w(0)=0$. Here $H_+$ is the set of sequences which are positive away from the component $j=0$, that is
$H_+= \left\{u \, \in \, \ell^2 \, : u_j\geq 0, \ j\geq1 \right\}$.

\begin{remark}
If the initial condition $\underline{u} \in H_+$, then $u(t) \in H_+$ for all $t \in [0,T]$. We will make use of this positivity condition in the forthcoming computations.
\end{remark}

In this section, we will obtain results concerning solutions of \eqref{model}. Since they are provided for a fixed element  $\omega \in \Omega$, they are equivalent to pathwise results for solutions of the stochastic system \eqref{model_stochastic}.  We start by introducing the definition of solution for \eqref{model}.

\begin{definition}[Solution of system \eqref{model}]
We say  that $u$ is a \emph{solution} of system \eqref{model} on $[0,T]$ with initial condition $\underline{u} \in H_+$ and noise $w \in C([0,T],\mathbb{R})$ with $w(0)=0$, 
if $u$  satisfies system \eqref{model}, $u_j \in C([0,T], \mathbb{R})$  for all $j$ and $u(t) \in H_+$ for all $t \in [0,T]$. 
\end{definition}

In the sequel we will denote by $\| \cdot \|$ the $\ell^2$-norm and by $\| \cdot \|_{\infty}$ the sup-norm.

\subsection{Existence}\label{subsect:existence}

The aim of this subsection is to establish an existence result for system \eqref{model}.

\begin{theorem}[Existence]\label{pathwise_exist}
For every  $T \in [0,+\infty)$, $\sigma \in \mathbb{R}_+$, $c \in [1,3]$, $\underline{u} \in H_+$ and $w \in C([0,T],\mathbb{R})$ with $w(0)=0$, system \eqref{model} admits at least a solution. 
\end{theorem}

The idea of the proof consists in considering a truncated version of \eqref{model}, for which global existence is ensured by uniform energy estimates, and then taking the limit with standard arguments. 

\begin{proof}
We first introduce a finite dimensional truncation of \eqref{model}. For any $N \in \mathbb{N}$, we consider the first $N$ equations of \eqref{model}, that is
\begin{equation}\label{modelN}
  \left\{\begin{array}{l l }
\displaystyle{u^{(N)}_0(t)=u_0(0)-\int_0^tu^{(N)}_0(s)u^{(N)}_1(s)ds+\sigma w(t)} &\\
du^{(N)}_j(t)= \left[-2^{cj}u^{(N)}_j(t)u^{(N)}_{j+1}(t)+2^{c(j-1)} \left(u^{(N)}_{j-1}(t) \right)^2\right] dt & \mbox{ for } j=1,\dots,N\\
u_j^{(N)}(0) = \underline{u}_j
& 
\mbox{ for } j=1, \dots, N          \\
u^{(N)}_j(t)\equiv0 & \mbox{ for } j \geq N+1
\end{array}\right.
\end{equation}
with $t \in [0,T]$ and $\underline{u}^{(N)} \in H_+$. 
System (\ref{modelN}) satisfies the hypotheses of Cauchy-Lipschitz theorem and then admits a local solution $\left(u_n^{(N)} \right)_{n \in \mathbb{N}}$ on $[0,\delta]$, for a certain $\delta>0$. 
We start by showing that the component $u_0^{(N)}(t)$ in \eqref{modelN} is uniformly bounded in $N$. To do so, we consider equation
\begin{equation*}\label{u_0}
u_0^{(N)}(t)=u^{(N)}_0(0)-\int_0^tu^{(N)}_0(s)u^{(N)}_1(s)ds +\sigma w(t).
\end{equation*}
Without loss of generality, we may assume 
\[
\max_{t \in [0,\delta] } \left| u^{(N)}_0(t) \right|= u^{(N)}_0(t^*) > 0,
\] 
for some $t^* \in [0,\delta]$ (the case $u^{(N)}_0(t^*) < 0$ can be treated similarly by symmetry). Now, let $\bar{t} = \sup \left\{ t \leq t^*: u^{(N)}_0(t)=0 \right\}$; if the set $\left\{ t \leq t^*: u^{(N)}_0(t)=0 \right\}$ is empty, let $\bar{t}=0$. Then, we have 
\begin{align*}
u^{(N)}_0(t^*) & \leq \left\vert u^{(N)}_0(0) \right\vert -\int_{\bar{t}}^{t^*}u^{(N)}_0(s)u^{(N)}_1(s) ds + \sigma [w(t^*)-w(\bar{t})] \\
&\leq \left\vert u^{(N)}_0(0) \right\vert + 2\sigma \sup_{t\in[0,T]}|w(t)| .
\end{align*} 
Therefore, for every $t$, it holds
\begin{equation}\label{bound_u0}
\left| u_0^{(N)}(t) \right| \leq a\, ,\quad  \mbox{ with }  \quad a := \|\underline{u}\| + 2 \sigma \|w\|_{\infty} \, .
\end{equation}
Now, we apply \eqref{bound_u0} to get a bound for the energy associated with dyadic model \eqref{modelN}. We start by considering the inequality
\[
\left\| u^{(N)}(t) \right\|^2 = \sum_{j=0}^N \left( u_j^{(N)}(t) \right)^2 \leq a^2 + \sum_{j=1}^N \left( u^{(N)}_j(t) \right)^2 \,.
\]
It is easy to see that, by taking the derivative of the summation term in the right-hand side of previous formula, thanks to cancellations, we get
\[
  \hspace{0.5cm }\frac{d}{dt}\left [\sum_{j=1}^N \left( u_j^{(N)}(t) \right)^2\right ]=2 \left( u_0^{(N)}(t) \right)^2 u_1^{(N)}(t) \leq 2 a^2 u^{(N)}_1(t) \leq 2 a^2 \sqrt{\sum_{j=1}^N \left( u_j^{(N)}(t) \right)^2} ,
\]
from which it follows by comparison
\[
\sum_{j=1}^N \left( u_j^{(N)}(t) \right)^2\leq \left( a^2 t +\sqrt{\sum_{j=1}^N \left( u_j^{(N)}(0) \right)^2}\right)^2\leq \left(a^2t+a\right)^2.
\]
Summarizing, we conclude
\begin{equation} \label{estimateN}
\left\|u^{(N)}(t)\right\|^2\leq a^2 + \left( a^2t +a \right)^2\leq \left(a^2T+2a\right)^2.
\end{equation}
From this bound we get global existence of the solution of system \eqref{modelN}.\\
We are left to prove existence for the infinite dimensional system \eqref{model}. We will obtain the result by means of Ascoli-Arzel\`a theorem and a standard diagonal argument. \\
For every fixed $j$ and $t \in [0,T]$, it holds:
\begin{itemize}
\item[i)] Uniform boundedness of $\left( u_j^{(N)} (t) \right)_{N \in \mathbb{N}}$ in both $N$ and $t$:
\begin{equation*}
\left\vert u_j^{(N)}(t) \right\vert  \leq \left\|u^{(N)}(t)\right\| \leq  a^2 T + 2a;
\end{equation*}
\item[ii)] Equi-Lipschitzianity of $\left( u_j^{(N)} (t) \right)_{N \in \mathbb{N}}$ with respect to $N$: by i), we have
\begin{equation*}
\left\vert \frac{d}{dt} u^{(N)}_j(t) \right\vert \leq 2^{cj+1} \, \left\|u^{(N)}(t)\right\|^2 \leq  2^{cj+1} \, \left( a^2 T + 2a \right)^2.
\end{equation*}
\end{itemize}
Ascoli-Arzel\`a theorem implies for each fixed $j$ the existence of a convergent subsequence in $C([0,T])$; i.e., it is possible to find indeces $\left\{ N^j_k, k \in \mathbb{N} \right\}$ such that
\[
\sup_{t \in [0,T]} \left\vert u_j^{(N^j_k)}(t) - u_j(t) \right\vert \, \stackrel{k\uparrow\infty}{\longrightarrow} \, 0 \quad \mbox{ for fixed $j$.}
\] 
The sequences $N_\bullet^j$ can be chosen so that $N_\bullet^{j+1}$ is a subsequence of $N_\bullet^j$ itself. By a standard diagonal argument we can extend the convergence to all $j$. Indeed, if we consider indeces $N_k := N_k^k$, we are extracting a common sub-subsequence such that 
\begin{equation*}
\sup_{t \in [0,T]} \left\vert u_j^{(N_k)}(t) - u_j(t) \right\vert \, \stackrel{k\uparrow\infty}{\longrightarrow} \, 0 \qquad \mbox{ for all } j\geq0.
\end{equation*} 
By taking the limit in the integral representation of the solution of \eqref{modelN}, one can see that 
the uniform limit $u=(u_j)_{j\geq0}$ in $C([0,T],\mathbb{R})$ is indeed a solution of \eqref{model}.  
\end{proof}

We can exploit part of the proof of Theorem~\ref{pathwise_exist} to simply obtain an energy estimate analogous to \eqref{estimateN} but valid for \emph{all} the solutions of the original system \eqref{model}.

\begin{proposition}[Energy bound]
For every $T \in [0,+\infty)$, $\sigma \in \mathbb{R}_+$, $c \in [1,3]$, $\underline{u} \in H_+$ and $w \in C([0,T],\mathbb{R})$ with $w(0)=0$, there exists $K_1=K_1(\sigma\|w\|_{\infty}, \|\underline{u}\|, T)$ that depends polynomially on  $\sigma\|w\|_{\infty}, \|\underline{u}\|$ and $T$ and such that any solution $u$ of system \eqref{model} satisfies the energy estimate
\begin{equation}\label{bound_norma}
\|u(t)\|^2\leq K_1 \ \ \ \mbox{ for all } t \in [0,T].
\end{equation}  

\end{proposition}

\begin{proof}
Consider a solution $u$ of system \eqref{model} and, for any $N \in \mathbb{N}$, let $u_0, u_1, \dots, u_N$ its first $N+1$ components. Notice that, by repeating the arguments as to derive \eqref{estimateN}, we get
\[
\vert u_0(t) \vert \leq a\, ,\quad  \mbox{ with }  \quad a := \|\underline{u}\| + 2 \sigma \|w\|_{\infty} 
\]
and
\[
   \hspace{1cm} \frac{d}{dt} \left[ \sum_{j=1}^N u_j^2(t) \right ] = 2 u_0^2(t) u_1(t) - 2^{c(N+1)} u_N^2(t) u_{N+1}(t) \leq 2 a^2 \sqrt{\sum_{j=1}^N u_j^2(t)},
\]
from which it follows
\begin{equation}\label{estimate}
\sum_{j=0}^N u_j^2(t) \leq \left( a^2 T + 2a \right)^2 \quad \mbox{ for every } N \geq 0.
\end{equation}
From \eqref{estimate}, by taking the limit as $N \to +\infty$, we conclude.
\end{proof}

It is worth to mention that the most part of following results heavily relies on \eqref{bound_norma}.

\subsection{Regularity}\label{subsect:regularity}

We aim at proving continuity of the solution $u$ of \eqref{model} with respect to the initial condition and to the function $w$ (see Theorem~\ref{thm:continuity} in Subsection~\ref{subsect:continuity}). To this purpose, the following regularity result will be crucial.

\begin{theorem}[Regularity]\label{pathwise_regularity}

For every $T \in [0,+\infty)$, $\sigma \in \mathbb{R}_+$, $c \in [1,3]$, $\underline{u} \in H_+$ and $w \in C([0,T],\mathbb{R})$ with $w(0)=0$, there exists $K_2=K_2(\sigma\|w\|_{\infty}, \|\underline{u}\|, T)$ that depends polynomially on  $\sigma\|w\|_{\infty}, \|\underline{u}\|$ and $T$ and such that any solution $u$ of system \eqref{model} satisfies the regularity condition
\begin{equation}\label{regularity}
\int_0^Tu_{j}^2(s)ds\leq K_2 \, 2^{-\frac{2}{3}cj}\, , \ \forall j \geq 0\,.
\end{equation}

\end{theorem}

In \cite{FrG-HVi} authors establish that a statistically stationary martingale solution $\bar{u}$ of \eqref{model_stochastic} satisfies the bound $\mathbb{E} ( \bar{u}_j^2 ) \leq k \, 2^{-\frac{2}{3}cj} $, with $k$ positive constant. To prove Theorem \ref{pathwise_regularity} we adapt the method they used to derive such an estimate. Our proof relies on the idea of replacing the average by an integral over the interval $[0,T]$. This trick allows to get regularity properties of pathwise  type.

Before proving Theorem \ref{pathwise_regularity}, we need the following technical lemma on real sequences.

\begin{lemma}\label{lmm:sequence}
Let $(a_n)_{n \in \mathbb{N}}$ be a sequence of positive real numbers such that $\limsup_{n\rightarrow+\infty} a_n<+\infty$. Suppose that, for some constants $\lambda > 0$, $0<C_1<1$ and $C_2 > 0$, there exists $\hat{n}$ such that, for $n \geq \hat{n}$, either
\begin{itemize}
\item[i)]
 $a_n\leq C_1 a_{n+1}+C_2 2^{-\lambda n}$
 \end{itemize}
 or
 \begin{itemize}
 \item[ii)] $a_n\leq C_1 a_{n+2}+C_2 2^{-\lambda n}$
 \end{itemize}
is satisfied. Then, for any $n \geq \hat{n}$, it holds
 \[
 a_n \leq \frac{C_2}{1-C_1} \, 2^{-\lambda n}.
 \]
\end{lemma}

\begin{proof} We start by proving the assertion under hypothesis i). Suppose by contradiction that there exists $\bar{n} \geq \hat{n}$ such that $a_{\bar{n}} >  \frac{C_2}{1-C_1} \, 2^{-\lambda \bar{n}}$.
%
Consider the inequality in i) and divide both sides by $C_1 \, a_{\bar{n}}$. We get
$$
\frac{a_{\bar{n}+1}}{a_{\bar{n}}}\geq \frac{1}{C_1}-\frac{C_2 \, 2^{-\lambda \bar{n}}}{C_1 \, a_{\bar{n}}}>\frac{1}{C_1}-\frac{1-C_1}{C_1} = 1.
$$ 
Since $a_{\bar{n}+1} > a_{\bar{n}}$, if we repeat the reasoning for the next ratio, we get
\[
\frac{a_{\bar{n}+2}}{a_{\bar{n}+1}} \geq \frac{1}{C_1}-\frac{C_2 \, 2^{-\lambda (\bar{n}+1)}}{C_1 \, a_{\bar{n}+1}} > \frac{1}{C_1}-\frac{C_2 \, 2^{-\lambda \bar{n}}}{C_1 \, a_{\bar{n}}} > 1.
\]
Proceeding by induction on $j$, we obtain that the sequence $(a_{\bar{n}+j})_{j \in \mathbb{N}}$ is monotonically increasing in $j$ and moreover, for all $j$, it holds
\[
\frac{a_{\bar{n}+j+1}}{a_{\bar{n}+j}} \geq  \frac{1}{C_1}-\frac{C_2 \, 2^{-\lambda \bar{n}}}{C_1 \, a_{\bar{n}}} =: k > 1,
\] 
where $k$ is a constant independent of $j$. Previous inequality
implies that the sequence $(a_n)_{n \in \mathbb{N}}$ diverges, which is a contradiction.\\
If we assume hypothesis ii) instead, the result follows by repeating the same argument as above for odd (or even) subsequences.
\end{proof}

\begin{proof}[Proof of Theorem~\ref{pathwise_regularity}]
For $T \in [0,+\infty)$ and for all $j\geq1$, a solution $u$ of system \eqref{model} satisfies
\begin{equation}\label{diff:u}
u_j(T)-u_j(0)=2^{c(j-1)}\int_0^T u^2_{j-1}(s)ds-2^{cj}\int_0^T u_j(s)u_{j+1}(s)ds
\end{equation}
and
\begin{equation}\label{diff:u2}
  \hspace{0.7cm} u^2_j(T)-u^2_j(0)=2\cdot2^{c(j-1)}\int_0^T u^2_{j-1}(s)u_j(s)ds-2\cdot2^{cj}\int_0^T u^2_j(s)u_{j+1}(s)ds.
\end{equation}
We want to estimate the terms $\int_0^T u_j^2(s) \, ds$ and $\int_0^T u_j^2(s)u_{j+1}(s)\, ds$ for every $j\geq0$. Starting from equation \eqref{diff:u}, we obtain
\begin{align}\label{stima:1}
  \int_0^T u^2_{j-1}(s)ds & = 2^{-c(j-1)} \left[u_j(T)-u_j(0)\right] + 2^{c}\int_0^T u_j(s)u_{j+1}(s)ds \nonumber\\
& \leq k_1 \, 2^{-cj} + k_2 \left[\int_0^T u^2_j(s)u_{j+1}(s)ds\right]^{1/2}\left[\int_0^T u^2_{j+1}(s)ds\right]^{1/4},
\end{align}
where $k_1$, $k_2$ are polynomials depending on $\sigma\|w\|_{\infty}$, $\|\underline{u}\|$ and $T$. The last inequality follows from the energy bound \eqref{bound_norma} and by applying twice H\"older inequality. Now we need an estimate for $\int_0^T u_j^2(s)u_{j+1}(s) \, ds$. Summing up terms in \eqref{diff:u2} from $j=1$ to $j=N$, we obtain
\[
  \hspace{1cm} 2^{cN+1} \int_0^T u^2_N(s)u_{N+1}(s) \, ds = 2 \int_0^T u_0^2(s)u_1(s) \, ds - \sum_{j=1}^N \left[u^2_j(T)-u^2_j(0)\right] \leq k,
\]
again by \eqref{bound_norma}. Therefore, for all $j\geq1$, we have 
\begin{equation}\label{stima:2}
\int_0^T u^2_j(s)u_{j+1}(s) \, ds \leq k 2^{-cj},
\end{equation}
with $k$ positive and independent of $j$. By using estimate \eqref{stima:2} in \eqref{stima:1} and then applying Young inequality, we get
\begin{align*}
  \int_0^T u_{j-1}^2(s)ds & \leq k_1 \, 2^{-cj} + k_2 \left[\int_0^T u^2_j(s)u_{j+1}(s)ds\right]^{1/2}\left[\int_0^T u^2_{j+1}(s)ds\right]^{1/4}\\
&\leq  k \, 2^{-\frac{2}{3}cj}
+\frac{1}{4}\int_0^T u^2_{j+1}(s) \, ds.
\end{align*}
The constants $k$'s appearing in the previous calculations may change from line to line, but always keep their polynomial nature. Since  $k$ does not depend on $t$ and the energy bound \eqref{bound_norma}  holds, we can apply Lemma~\ref{lmm:sequence} to conclude~\eqref{regularity}.
\end{proof}

\subsection{Continuity}\label{subsect:continuity}

Next theorem is a result of continuity with respect to initial condition and noise. This will imply, on the one hand, uniqueness of solution $u$ of system~\eqref{model}; on the other, it will guarantee existence and uniqueness of strong solutions for system \eqref{model_stochastic}.\\

Before giving the result we need some more notation. For any $\alpha \in \mathbb{R}$, we denote by $H^{\alpha}$ the Sobolev-type space
$$
H^{\alpha} := \left\{ u=(u_n)_{n\in\mathbb{N}}: \left\| {u} \right\|^2_{\alpha} < \infty \right\},
$$
with norm $\| \cdot \|_{\alpha}$ given by
$
\| u \|_{\alpha}^2 :=\sum_{j=0}^{\infty}2^{2\alpha j } u_j^2 \,.
$
Notice that $H^0=\ell^2$. 

\begin{theorem}[Continuity]\label{thm:continuity}
Let $c\in[1,3)$, $T \in [0,+\infty)$, $\sigma \in \mathbb{R}_+$, $\underline{u} \in H_+$ , $w \in C([0,T],\mathbb{R})$ with $w(0)=0$. There exists a function $f:\mathbb{R}_+ \longrightarrow \mathbb{R}_+$ with $\lim_{\delta\rightarrow0}f(\delta)=0$ such that, for all $\underline{\tilde{u}}$ $\in$ $H_+$ and all $\tilde{w}$ $\in$ $C([0,T],\mathbb{R})$ with $\tilde{w}(0)=0$, if
\begin{equation*}
 \left\| \underline{u}-\underline{\tilde{u}} \right\| < \delta \quad \mbox{ and } \quad \left\| w-\tilde{w} \right\|_{\infty} < \delta \, ,
\end{equation*}
then
\begin{equation*}
\sup_{t\in [0,T]}\|u(t)-\tilde{u}(t)\|_{-\frac{1}{2}}<f(\delta) \, ,
\end{equation*}
where $u$ (resp. $\tilde{u}$) is a solution of system \eqref{model} with initial condition 
$\underline{u}$ (resp. $\underline{\tilde{u}}$) and noise $w$ (resp. $\tilde{w}$). Moreover, if $w=\tilde{w}$ and $\underline{u}\neq\underline{\tilde{u}}$, then
\begin{equation}\label{contrazione}
\| u(t)-\tilde{u}(t) \|_{-\frac{1}{2}}<\| u(0)-\tilde{u}(0) \|_{-\frac{1}{2}}.
\end{equation}  
\end{theorem}

Estimate \eqref{contrazione} is a contraction property  somehow unusual for models deriving from fluid dynamics. Nevertheless, in \cite{Bor13} it is possible to find an analogous estimate for the $L^1$-norm of the viscous Burgers model. Similarly to \cite{Bor13}, we use \eqref{contrazione} to study statistically stationary distributions.
\begin{proof}
We set
\[
y_j (t) :=u_j(t) + \tilde{u}_j(t) \quad \mbox{ and } \quad z_j (t) := u_j(t) - \tilde{u}_j(t).
\]
From \eqref{model}, we obtain a system of equations for $z_j$. It is readily seen that
\begin{equation}\label{eqs:Z}
  \left\{
\begin{array}{lc}
z_0(t) = \displaystyle{z_0(0)-\frac{1}{2}\int_0^t \left[z_0(s)y_1(s)+z_1(s)y_0(s) \right]ds+\sigma \left[w(t)-\tilde{w}(t) \right]} & \\
z_j(t) = \displaystyle{z_j(0)+2^{c(j-1)} \int_0^t z_{j-1}(s) y_{j-1}(s) \, ds} & \\
\phantom{z_j(t) =}  \qquad \displaystyle{- \frac{2^{cj}}{2} \int_0^t \left[ z_j(s) y_{j+1}(s) + z_{j+1}(s) y_j(s) \right] ds} & \mbox{ for } j \geq 1 \\
z_j(0)=\underline{u}_j-\underline{\tilde{u}}_j & \mbox{ for } j \geq 0 \\ 
\end{array}
\right.
\end{equation}
We borrow a trick from \cite{BaFlMo10} and we study the $H^{-\frac{1}{2}}$-norm of $z(t)$. We aim at getting an upper bound for 
\begin{align*}
\psi_n (t) &:= \sum_{j=0}^n \frac{z_j^2(t)}{2^j}\\
&=\phi_n(t)+2\sigma z_0(0)\left[w(t)-\tilde{w}(t)\right]+\sigma^2 \left[w(t)-\tilde{w}(t)\right]^2\\
&\qquad \qquad - \sigma \left[ w(t)-\tilde{w}(t)\right] \int_0^t \left[z_0(s)y_1(s)+z_1(s)y_0(s)\right]ds \,,
\end{align*}
where 
\begin{align*}
\phi_n(t) &:= z^2_0(0)+\frac{1}{4} \left( \int_0^t \left[z_0(s)y_1(s)+z_1(s)y_0(s)\right] ds\right)^2 \\
&\qquad\qquad -z_0(0)\int_0^t \left[z_0(s)y_1(s)+z_1(s)y_0(s)\right] ds +\sum_{j=1}^n\frac{z_j^2(t)}{2^j} \, . 
\end{align*}
First we study the quantity $\phi_n$. By using equations \eqref{eqs:Z}, we compute the derivative
\begin{align*}
\frac{d}{dt} \phi_n (t) &=\frac{1}{2} \left[z_0(t)y_1(t)+z_1(t)y_0(t)\right] \int_0^t \left[z_0(s)y_1(s)+z_1(s)y_0(s)\right]ds \\
&\qquad \qquad -z_0(0)[z_0(t)y_1(t)+z_1(t)y_0(t)] + \sum_{j=1}^n \frac{2}{2^j} \, z_j (t) \, \frac{d}{dt} z_j (t) \\
&= \left[z_0(t)y_1(t)+z_1(t)y_0(t)\right]\left\{-z_0(t)+z_0(0)+\sigma \left[w(t)-\tilde{w}(t)\right]\right\}   \\
&\qquad \qquad -z_0(0)[z_0(t)y_1(t)+z_1(t)y_0(t)] + \sum_{j=1}^n \frac{2}{2^j} \, z_j (t) \, \frac{d}{dt}z_j (t) \\
&= -\sum_{j=0}^n2^{(c-1)j}z_j^2(t)y_{j+1}(t)-2^{(c-1)n}z_n(t)z_{n+1}(t)y_n(t) \\
& \qquad \qquad +\sigma \left[w(t) -\tilde{w}(t)\right] \left[z_0(t)y_1(t)+z_1(t)y_0(t)\right].
\end{align*}
Now we estimate $\frac{d}{dt}\phi_n$. We observe that the first term is negative, since $y$ $\in$ $H_+$. Then, we are left to consider
\begin{align*}
\frac{d}{dt} \phi_n (t) &\leq - 2^{(c-1)n} \, z_n (t) \, z_{n+1} (t) \, y_{n} (t) +\sigma \left[ w(t)-\tilde{w}(t) \right] \left[z_0(t)y_1(t)+z_1(t)y_0(t)\right] \\
&\leq - 2^{(c-1)n} \, z_n (t) \, z_{n+1} (t)y_n(t)+2\sigma \delta \left(\|u(t)\|^2+\|\tilde{u}(t)\|^2\right) \\
&\leq - 2^{(c-1)n} \, z_n (t) \, z_{n+1} (t)y_n(t)+\delta k  \,,
\end{align*} 
where constant $k$ follows from \eqref{bound_norma}. In particular, $k$ is a polynomial in $\sigma \|w\|_{\infty}$, $\sigma \|\tilde{w}\|_{\infty}$, $\| \underline{u} \|$, $\|\underline{\tilde{u}} \|$ and $T$. In the sequel, the value of constants may change from line to line, but they are always polynomial functions of those quantities. By integrating the previous inequality, we get
\begin{align*}
  \phi_n (t) &   \leq  \phi_n(0) -2^{(c-1)n}\int_0^t [u_n(s)-\tilde{u}_n(s)][u_{n+1}(s)-\tilde{u}_{n+1}(s)][u_n(s)+\tilde{u}_n(s)]ds + \delta k T \\ 
%
%
&   \leq 2\delta^2+ k_1 \, 2^{(c-1)n} \left(\int_0^tu_n^2(s)ds+\int_0^t\tilde{u}_n^2(s)ds\right)+\delta k  T \\
%
%
%
%
&   \leq 2\delta^2+k_1 \, 2^{(\frac{c}{3}-1)n} +\delta k T, 
%
\end{align*}
where the last inequality is due to the regularity condition \eqref{regularity}. Then, for $\psi_n(t)$ it holds
\begin{equation*}
\psi_n(t)\leq  k_1 \, 2^{(\frac{c}{3}-1)n} + k_2 \delta T +k_3 \delta^2. 
%
\end{equation*}
Thus, if $c<3$ we obtain 
\begin{equation*}
\|u(t)-\tilde{u}(t)\|^2_{-\frac{1}{2}}=\lim_{n\rightarrow +\infty} \psi_n(t) \leq \underbrace{k_2 \delta T +k_3 \delta^2}_{\mbox{\footnotesize $=: f(\delta)$}} \, ,
\end{equation*}
which is arbitrary small for small $\delta$, as wanted.

Now let $u$ and $\tilde{u}$ be two  solutions of system \eqref{model} with the same noise $w$, but with different initial conditions $\underline{u}$ and $\underline{\tilde{u}}$, respectively. It is easy to verify that, if we perform the same computations as above in the case $w=\tilde{w}$, we get the following identity
\begin{align*}
\frac{d}{dt} \psi_n (t) &= \sum_{j=0}^n \frac{2}{2^j} \, z_j (t) \, \frac{d}{dt}z_j (t) \\
&= - \sum_{j=0}^n 2^{(c-1)j} \, z_j^2 (t) \, y_{j+1} (t) - 2^{(c-1)n} \, z_n (t) \, z_{n+1} (t) \, y_{n} (t).
\end{align*}
By integrating over the interval $[0,t]$, it yields 
\begin{equation}\label{12}
  \psi_n (t)-\psi_n (0)=- \! \int_0^t \! \sum_{j=0}^n 2^{(c-1)j} \, z_j^2 (s) y_{j+1} (s)ds - \! \int_0^t \! 2^{(c-1)n} \, z_n (s) z_{n+1} (s) y_{n} (s) ds.
\end{equation}
The first term on the right-hand side of \eqref{12} is strictly negative and decreasing as $n$ goes to infinity. From this and taking the limit in $n$, it immediately follows 
\begin{equation*}
\| u(t)-\tilde u(t) \|_{-\frac{1}{2}}<\| u(0)-\tilde u(0) \|_{-\frac{1}{2}}.
\end{equation*} 
\end{proof}
A consequence of the previous theorem is the uniqueness of the solution of~\eqref{model}.

\begin{corollary}[Uniqueness]\label{pathwise_uniqueness} Let $u$ and $\tilde{u}$ be two solutions of system~\eqref{model} with the same initial condition $\underline{u} \in H_+$ and the same noise $w \in C([0,T], \mathbb{R})$ with $w(0)=0$, then
\[
u(t)=\tilde{u}(t) \ \ \ \forall  t \in [0,T].
\]
\end{corollary}

\section{Strong solutions}\label{sect:strong:sol}

In the first part of this section we prove that a strong solution for problem \eqref{model_stochastic} exists and is unique. 
We then tackle the question of stationary distributions. In this respect, we establish existence of a strong statistically stationary solution in $L^2 (\Omega, H^{\alpha} \cap H_+)$, for every $\alpha < \frac{c}{3}$, and further we show uniqueness in $L^2 (\Omega, H_+)$. All the results are valid in the range of $c \in [1,3)$.\\  
We start by giving the definition of strong solution.

\begin{definition}[Strong solution]  
 Let $\left( \Omega, \mathcal{F}, \{\mathcal{F}_t\}_{t \in [0,T]},\mathbb{P} \right)$ be a filtered probability space and $W$ a Brownian motion on $(\Omega, \mathcal{F},\mathbb{P})$, with respect to the filtration $\{\mathcal{F}_t\}_{t \in [0,T]} $. We say that $u$ is a \emph{strong solution} of system \eqref{model_stochastic} with $\mathcal{F}_0$-measurable initial condition   $\underline{u}$, if $u$ is a stochastic process on $(\Omega, \mathcal{F},\mathbb{P})$ with continuous sample paths, satisfying \eqref{model_stochastic} and adapted to the filtration generated by $W$ and $\underline{u}$. \\
 
Moreover, this solution is \emph{unique} if, given two solutions $u^{(1)}$ and $u^{(2)}$ with the same initial condition, it holds $u^{(1)}(t)=u^{(2)}(t)$ for all $t \in [0,T]$ a.s.. 
\end{definition}

Existence and uniqueness of a pathwise solution, together with continuity with respect to the noise and the initial data, will allow to readily get existence and uniqueness in the strong sense. \\
By Theorem~\ref{pathwise_exist}, given an $\mathcal{F}_0$-measurable positive initial condition $\underline{u}$ and the noise, we can construct an application 
\begin{equation}\label{application:F}
\begin{array}{lccc}
F_T: & \ell^2 \times C([0,T],\mathbb{R}) &\longrightarrow & C \left( [0,T],\ell^2 \right) \\
       & (\underline{u}(\omega), W(\cdot,\omega)) &\longmapsto &  F_T(\underline{u}(\omega), W(\cdot,\omega)), 
\end{array}
\end{equation}
that associates the pair $(\underline{u}(\omega), W(\cdot,\omega))$ with the corresponding solution
 of system~\eqref{model} on $[0,T]$. For every $T \in [0,+\infty)$, the function $F_T$ is well-defined thanks to Corollary~\ref{pathwise_uniqueness} and it is continuous with respect to the initial datum and the noise by Theorem~\ref{thm:continuity}.  
%
%
As a consequence, the map $u(\cdot,\omega) := F_T(\underline{u}(\omega), W(\cdot,\omega))$ from $\Omega$ to $C \left( [0,T],\ell^2 \right)$ is measurable with respect to the $\sigma$-field generated by $\underline{u}$ and the Brownian motion. Concluding, $u (\cdot, \omega)$ is a strong solution of \eqref{model_stochastic}.
Such a solution is unique by Corollary \ref{pathwise_uniqueness}. Hence, we obtain:

\begin{theorem}[Existence and uniqueness]\label{thm:strong:ex}
For every $T \in [0,+\infty)$, $\sigma \in \mathbb{R}_+$, $c \in [1,3)$ and initial condition $\underline{u}$, $\mathcal{F}_0$-measurable random variable with values in $H_+$, system \eqref{model_stochastic} admits a unique strong solution $u$. 
\end{theorem}

With the above results in hand, we can now turn to the analysis of strong statistically stationary solutions.

\begin{definition}[Strong stationary solution]
We say that $u^*$  is  a \emph{strong stationary solution} of system \eqref{model_stochastic}, if it is a strong solution and the distribution of $u^*(t)$ does not depend on $t$, for all $t \in [0,T]$. 
\end{definition}

\begin{remark}
Weak existence of a statistically stationary solution $u^*$ for system \eqref{model_stochastic} is proved in  \cite{FrG-HVi}. Furthermore, by \cite[Prop. 3.1 and Thm. 4.2]{FrG-HVi}, $u^*$ can be chosen so that it belongs to $H_+$ for all times a.s..
\end{remark}

Existence and uniqueness of a strong solution imply, by Yamada and Watanabe theorem, uniqueness in law for solutions of system \eqref{model_stochastic}. This plus weak existence of a statistically stationary solution in $H_+$ (see previous Remark) gives strong existence. Moreover, from a moment estimate in \cite[Thm. 4.2]{FrG-HVi} we can infer that such a solution belongs to $L^2 (\Omega,H^{\alpha} \cap H_+)$, for every $\alpha < \frac{c}{3}$. Thus, the following theorem remains proved. 

\begin{theorem}[Existence of a  strong stationary solution] \label{thm:strong:stationary_ex}
For every $T \in [0,+\infty)$, $\sigma \in \mathbb{R}_+$ and $c \in [1,3)$, system \eqref{model_stochastic} admits a strong stationary solution $u^*$ such that $u^* \in L^2 \left(\Omega,H^{\alpha} \cap H_+ \right)$, for every $\alpha < \frac{c}{3}$.
\end{theorem}

To conclude our analysis, we prove that only one strong statistically stationary solution may exist. 

\begin{theorem}[Uniqueness of stationary distribution] \label{thm:strong:stationary_uni}
For every $T \in [0,+\infty)$, $\sigma \in \mathbb{R}_+$ and $c \in [1,3)$, if $u_1^*, u_2^* \in L^2(\Omega, H_+)$ are strong stationary solutions of \eqref{model_stochastic}, then $u^*_1$ and $u^*_2$ have the same distribution. 
\end{theorem}

The proof relies on the Kantorovich's formulation for the optimal transport problem. We will introduce a distance between probability measures as a cost function to be minimized. Then, we will use formula \eqref{contrazione} to show that the existence of two different stationary distributions would contradict the minimality achieved by the optimal transport plan.

\begin{proof}
By contradiction, let $\mu_1$ and $\mu_2$ be two different stationary distributions  in $H_+$ giving rise to strong stationary solutions of \eqref{model_stochastic}. We can define the $2$-plans with marginals $\mu_1$ and $\mu_2$ as 
$$
\Gamma(\mu_1,\mu_2):=\left\{\eta \, \in \, \mathcal{M}_1(\ell^2\times \ell^2)  : \int_{\ell^2}\eta(x,dy)=\mu_1 \mbox{ and } \int_{\ell^2}\eta(dx,y)=\mu_2 \right\}.
$$
Moreover, for every $\eta \in \Gamma(\mu_1, \mu_2)$, we introduce the functional 
$$
\varphi (\eta):= \int_{\ell^2\times \ell^2} \| x-y\|^2_{-\frac{1}{2}}\, \eta(dx,dy).
$$
%
%
By Kantorovich's result on optimal transport problem, we know there exists an element $\eta_0$ in $\Gamma(\mu_1,\mu_2)$, such that 
\begin{equation}\label{Kantorovich}
\varphi (\eta_0) \leq \varphi (\eta),  \ \forall  \eta \in \Gamma(\mu_1,\mu_2).
\end{equation}
On $\tilde{\Omega}:=\ell^2\times \ell^2\times \Omega$, where  $\Omega$ is the sample space of Brownian motion, we construct the random vector $(u^*_1(0), u^*_2(0))$ with joint law $\eta_0$, the probability measure that realizes the minimum \eqref{Kantorovich}.
Let  $u^*_1$ (resp.  $u^*_2$ ) be the strong stationary solution with initial condition $u^*_1(0)\sim \mu_1$ (resp. $u^*_2(0)\sim \mu_2$). After a time $t>0$, the random vector $(u^*_1(t), u^*_2(t))$ will have joint law $\eta_t$ that, by stationarity, is still belonging to $\Gamma(\mu_1,\mu_2)$. Therefore, from \eqref{Kantorovich}, we get  
\begin{equation}\label{contraddizione}
  \hspace{0.5cm}\begin{array}{ccccc}
\varphi (\eta_0) && \leq &&   \varphi (\eta_t)\\
\verteq &&&&\verteq \\
\displaystyle{\int_{\tilde{\Omega}}\|u^*_1(0)-u^*_2(0)\|^2_{-\frac{1}{2}}d\tilde{\mathbb{P}}(\tilde{\omega}) }&&& &\displaystyle{\int_{\tilde{\Omega}}\|u^*_1(t)-u^*_2(t)\|^2_{-\frac{1}{2}}d\tilde{\mathbb{P}}(\tilde{\omega})}
\end{array}
\end{equation}
with $\tilde{\mathbb{P}}$ probability measure on $\tilde{\Omega}$. On the other hand, by taking expectation on both sides of \eqref{contrazione}, it yields 
\begin{equation*}
\displaystyle{\int_{\tilde{\Omega}}\|u^*_1(0)-u^*_2(0)\|^2_{-\frac{1}{2}}d\tilde{\mathbb{P}}(\tilde{\omega}) }\quad > \quad \displaystyle{\int_{\tilde{\Omega}}\|u^*_1(t)-u^*_2(t)\|^2_{-\frac{1}{2}}d\tilde{\mathbb{P}}(\tilde{\omega})},
\end{equation*}
that contradicts \eqref{contraddizione}.
\end{proof}

\section*{Acknowledgments}

The authors wish to warmly  thank Markus Fischer for valuable discussions and suggestions. DB has been partially supported  by the University of Padova  through the Project ``Stochastic Processes and Applications to Complex  Systems'' (CPDA123182). FC acknowledges financial support of FIRB research grant RBFR10N90W. MF has been partially supported by GA\v{C}R grant P201/12/2613. LP acknowledges financial support of the research project  ``Singular perturbation problems for differential operators'', Progetto di Ateneo of the University of Padova.

\bibliographystyle{abbrv}


\end{document}